\documentclass[a4paper,12pt]{article}
\usepackage[top=2.5cm,bottom=2.5cm,left=3cm,right=3cm]{geometry}
\usepackage{amsmath,amsthm,mathrsfs,lineno}
\usepackage{txfonts,pxfonts,upgreek} 
\usepackage{latexsym}
\usepackage{graphicx,booktabs,multirow}
\usepackage{tikz}
\usetikzlibrary{bbox}
\usepackage{appendix}
\usepackage{yhmath}
\usepackage{float}
\usepackage{color}
\usepackage{subfigure}
\usepackage{rotating}
\usepackage[section]{placeins}%figure in one section
\usepackage{setspace}
\usepackage{pdfpages}

\newtheorem{theorem}{Theorem}
\newtheorem{lemma}{Lemma}

\newtheorem{observation}{Observation}
\newtheorem{problem}{Problem}
\newtheorem{proposition}{Proposition}

\newtheorem{claim}{Claim}[proposition]

\newtheorem{definition}{Definition}
\newtheorem{case}{Case}
\newtheorem{remark}{Remark}

\AddToHook{env/proof/begin}{\setcounter{case}{0}} %情形重新编号

\usepackage[colorlinks=true,anchorcolor=blue,filecolor=blue,linkcolor=blue,urlcolor=blue,citecolor=blue]{hyperref}
\usepackage{url}

\numberwithin{equation}{section}

\usepackage{pdfpages} %!PDFLATEX
\usepackage{caption}

\usetikzlibrary{backgrounds}
\usetikzlibrary{arrows}
\usetikzlibrary{shapes,shapes.geometric,shapes.misc}
\pgfdeclarelayer{edgelayer}
\pgfdeclarelayer{nodelayer}
\pgfsetlayers{background,edgelayer,nodelayer,main}
\tikzstyle{none}=[inner sep=0mm]
\tikzstyle{bluenode}=[fill=blue, draw=black, shape=circle, minimum size=0.2cm, 
inner sep=0.2pt]
\tikzstyle{blacknode}=[fill=black!50, draw=black, shape=circle, minimum 
size=0.2cm, inner sep=0.2pt]
\tikzstyle{rednode}=[fill={rgb,255: red,244; green,0; blue,0}, draw=black, 
shape=circle, minimum size=0.2cm, inner sep=0.2pt]
\tikzstyle{square}=[draw=black, shape=rectangle, minimum 
size=0.2cm, inner sep=0.2pt]

\tikzstyle{blackedge}=[-, draw=black, fill=none, line width=0.15mm]
\tikzstyle{blueedge}=[-, draw=blue, fill=none, line width=0.3mm]
\tikzstyle{blackedge_thick}=[-, draw=black, opacity=0.9, line width=0.45mm, fill=none]
\tikzstyle{black_thick}=[-, draw=black!80,opacity=0.9, line width=0.5mm, fill=none]
\tikzstyle{rededge}=[-, draw=red!80,opacity=0.8]
\tikzstyle{rededge_thick}=[-, line width=0.5mm, opacity=0.9, draw=red!75]
\tikzstyle{blue_thick_opacity}=[-, fill opacity=0.7, draw=blue, 
fill=none, line width=0.4mm, opacity=0.7]
\tikzstyle{blue_thick}=[-, line width=0.5mm, draw=blue]
\newcommand{\proofend}{{\hfill$\Box$}}

\let\oldbibliography\thebibliography
\renewcommand{\thebibliography}[1]{%
  \oldbibliography{#1}%
  \setlength{\itemsep}{-2pt}%
  \setlength{\baselineskip}{13pt}
  \setlength{\lineskiplimit}{-\maxdimen}
}
\newcommand{\Rot}{\textnormal{Rot}}
\begin{document}

%\linenumbers
 \captionsetup[figure]{labelfont={bf},name={Fig.},labelsep=period}
	%\linenumbers 
	\title{The reducibility of optimal 1-planar graphs \thanks{E-mail addresses:  
\url{lczhangmath@163.com} (L. Zhang), \url{hyqq@hunnu.edu.cn.} (Y. Huang, corresponding author).
} }

% corresponding author
\author{Licheng Zhang $^{\dag}$,Yuanqiu Huang $^{\ddag}$ \\
{\small $^{\dag}$ School of Mathematics, Hunan  University} \\ 
 {\small Changsha, Hunan, 410082, P.R. China} \\ 
 {\small $^{\ddag}$School of Mathematics and Statistics, Hunan Normal University} \\
{\small Changsha, Hunan, 410081, P.R. China }  
}

%	\date{\today}
\date{}
\maketitle
%	Lexicographic product, as a standard binary operation on graphs, was 
%introduced by Felix Hausdorff in 1914. 
\begin{abstract}
 A graph is  reducible if it is the lexicographic product of two 
smaller non-trivial graphs.  It is well-known that a 1-planar  graph  with $n ~(\ge3)$ vertices has at most $4n-8$ edges, and a 1-planar graph $G$ with $n$ vertices is optimal if $G$  has exactly $4n-8$ edges. In this paper, we characterize the reducibility of optimal 1-planar graphs. This work is motivated by a problem posed by Bucko and Czap in 2015, which concerns determining the 1-planarity of the lexicographic product of a graph and two isolated vertices.

 %Moreover, we nvestigate the reducibility of 1-planar graphs, with a particular focus on 
%prove  that a reducible 1-planar graph 
%with $n$ vertices has at most $4n-9$ edges for $n=6$ or $n\ge 9$. We also prove 
%that this bound is sharp.  As a byproduct, two necessary conditions for graph $G\circ 	2K_1$ to be 1-planar are given. 

%	 This 
%is also a progress on an unsolved problem regarding the 1-planarity of 
%graph $G\circ 	2K_1$ posed by Bucko and 	Czap in 2015.


\end{abstract}

\section{Introduction}

\subsection{The study of 1-planarity of graphs}
All graphs considered here are simple and finite. Let $G=(V, E)$ be a graph. Its order and size are $|V|$ and $|E|$, respectively. A drawing of $G$ is a mapping $D$ that assigns to each vertex in $V$ a distinct point in the plane and to each edge $u v$ in $E$ a continuous arc connecting $D(u)$ and $D(v)$. 
A drawing of a graph is {\it $1$-planar} if each edge is 
crossed at most once. A graph is  {\it $1$-planar} if it has a 1-planar 
drawing, and such drawing is called {a \it $1$-plane graph}. 
  1-planar graphs were first studied by Ringel (1965) \cite{Ringel1965}. 
 Since then, many properties of 1-planar graphs have 
  been widely investigated; see \cite{Czap2013, Huang2022s, zhangxin2016} for  examples,  
  or \cite{kobourov2017} for  a  survey. 
  
Determining the 1-planarity of a graph is a  fundamental problem in the topic of 1-planar graphs. Unfortunately, determining the 1-planarity of a graph is NP-complete \cite{Korzhik2013},  and it has been pointed out in \cite{matsumoto2021} that 1-planar graphs cannot be characterized by using a finite set of forbidden (topological) minors, unlike the well-known criterion (Wagner's theorem and Kuratowski's theorem) for planar graphs. Based on these two main reasons, the 1-planarity of only a few graphs has been determined, and even for graphs with given a small order, their 1-planarity remains unknown \cite{binucci2020}.  
In addition,  few necessary conditions for 1-planarity are known.  Currently, researchers are focusing on specific classes of graphs (with well-defined structures), including operations on  graphs, to determine their 1-planarity. For examples, Czap and Hudák determined the 1-planarity of complete multipartite graphs \cite{czap2012}. Czap, Hudák and Madaras \cite{czap2014} studied the 1-planarity of the join of two graphs.

\begin{definition}
 A {\it lexicographic product} $G \circ H$ of two graphs $G$ 
and $H$ is a graph 
 such that the vertex set of $G\circ H$ is the Cartesian product $V (G)\times V 
 (H)$ and two vertices $(g_1,h_1)$ and $(g_2,h_2)$ are adjacent in $G\circ H$ 
 if and only if
 \begin{itemize}
   \item $g_1$ is adjacent to $g_2$ in $G$, or
   \item $g_1=g_2$ and $h_1$ 
 is adjacent to $h_2$ in $H$.
 \end{itemize}

\end{definition}

 In 2015, Bucko and Czap \cite{bucko2015} discussed the 
1-planarity of lexicographic products of graphs. The authors propose a conjecture:
A graph $G\circ K_2$ is $1$-planar if and only if $G$ is a \emph{cactus}, which is a connected graph in which every edge belongs to
at most one cycle.
Recently, Matsumoto and Suzuki \cite{matsumoto2021} confirmed the conjecture. Thus, combining the works \cite{bucko2015} and \cite{matsumoto2021}, the 1-planarity of the lexicographic product of two graphs leaves only the following unsolved problem.
\begin{problem}[\cite{bucko2015}]\label{openPro}
Let $G$ be a connected graph with maximum degree at least $3$. Characterize the 
$1$-planarity of  $G\circ 2K_1$.
\end{problem}

Note that the lexicographic product is not commutative \cite[Chapter 5, Page 56]{hammack2011}, i.e., $G \circ H \neq H \circ G$ in general. A graph is  {\it 
trivial} if it has just one vertex or no vertex.  A 
graph is called {\it reducible} if it is the lexicographic product of 
two non-trivial graphs, otherwise is called {\it irreducible}\footnote{ In general, determining whether a graph is reducible is as difficult as determining whether two graphs are isomorphic \cite{feigenbaum1986}.}. 
It is well-known that a 1-planar graph with $n(\geq 3)$ vertices has at most $4 n-8$ edges, and a 1-planar graph  with $n$ vertices is \emph{optimal} if it has exactly $4 n-8$ edges. This means that optimal 1-planar graphs are a class of 1-planar graphs with the maximum number of edges. In this paper, we focus on characterizing the reducibility of optimal 1-planar graphs. We shall show that, with the exception of a single graph $ K_{2,2,2,2}$, all optimal 1-planar graphs are not reducible. 
   As a byproduct, we provide the upper bound on the size of the left factor $G$ for the graph $G \circ 2 K_1$ to be 1-planar, which will serve as a necessary condition for Problem \ref{openPro}.

\begin{theorem}\label{main}
Any optimal $1$-planar graph $G$ is not reducible unless $G\cong K_{2,2,2,2}$. 
\end{theorem}

 \begin{figure}[H]
\centering
\begin{tikzpicture}[scale=0.4]
	\begin{pgfonlayer}{nodelayer}
		\node [style=blacknode] (0) at (-0.95, 4) {};
		\node [style=blacknode] (1) at (3.05, 4) {};
		\node [style=blacknode] (2) at (-0.95, 0) {};
		\node [style=blacknode] (3) at (3.05, 0) {};
		\node [style=blacknode] (4) at (-2.95, 6) {};
		\node [style=blacknode] (5) at (-2.95, -2) {};
		\node [style=blacknode] (6) at (5.05, 6) {};
		\node [style=blacknode] (7) at (5.05, -2) {};
	\end{pgfonlayer}
	\begin{pgfonlayer}{edgelayer}
		\draw [style=blueedge, in=65, out=25, looseness=2.20] (4) to (7);
		\draw [style=blueedge, in=155, out=115, looseness=2.20] (5) to (6);
		\draw [style=blueedge] (6) to (0);
		\draw [style=blueedge] (4) to (1);
		\draw [style=blueedge] (4) to (2);
		\draw [style=blueedge] (0) to (5);
		\draw [style=blueedge] (2) to (1);
		\draw [style=blueedge] (0) to (3);
		\draw [style=blueedge] (6) to (3);
		\draw [style=blueedge] (1) to (7);
		\draw [style=blueedge] (5) to (3);
		\draw [style=blueedge] (2) to (7);
		\draw [style={black_thick}] (4) to (6);
		\draw [style={black_thick}] (4) to (0);
		\draw [style={black_thick}] (4) to (5);
		\draw [style={black_thick}] (5) to (2);
		\draw [style={black_thick}] (2) to (3);
		\draw [style={black_thick}] (3) to (1);
		\draw [style={black_thick}] (1) to (0);
		\draw [style={black_thick}] (0) to (2);
		\draw [style={black_thick}] (5) to (7);
		\draw [style={black_thick}] (7) to (3);
		\draw [style={black_thick}] (1) to (6);
		\draw [style={black_thick}] (6) to (7);
	\end{pgfonlayer}
\end{tikzpicture}

\caption{A 1-planar drawing of $K_{2,2,2,2}$}
\label{K2222_1}
\end{figure}

\begin{theorem}\label{coro1}
Let $G$ be a graph of order $n\ge 5$. If $G\circ 2K_1$ is 
$1$-planar, then $|E(G)|\le 2n-3$.
\end{theorem}

In a certain sense, Theorem \ref{main} subtly suggests that optimal 1-planar graphs lack strong symmetry (from the perspective of the lexicographic product), as the structure of the graph $G \circ H$ implicitly encodes numerous copies of the right factor $H$ and complete bipartite subgraphs rooted in the left factor $G$.
On the other hand, the upper bound on the size of a reducible 1-planar graph is also naturally given. Interestingly, by reducing the number of edges by just one, there exist infinitely many reducible 1-planar graphs.

\begin{theorem}\label{coro0}
Let $G$ be a $1$-planar graph of order $n\ge 9$. 
If $G$ is reducible,  then $|E(G)|\le 4n-9$. Moreover, the bound for $n$ being divisible by 3  is 
tight.

\end{theorem}

\subsection{Terminology and notations}\label{notations}
 The terms not defined here can be found in \cite{west2001}.  Let  $K_n$ 
denote  the complete graph of order $n$, respectively. The 
disjoint union of $k$ copies of a graph $G$ is denoted by $kG$. A \emph{$k$-cycle } $C_k$ is a cycle of with $k$ vertices.

Let  $G$ be a graph.  If $S$ is a set of vertices of $G$, the 
vertex-induced subgraph $G[S]$ is the subgraph of $G$ that has $S$ as its set 
of vertices and contains all the edges of $G$ that have both end-vertices in 
$S$.  We denote by $G\cong H$ that graphs $G$ and $H$ are isomorphic. We 
denoted by $K_{n_1,n_2,\cdots, n_k }$ a complete multipartite graph with $k$ 
partition classes of sizes $n_1, 
n_2,\dots, n_k$, respectively.  We call a cycle $C$ of $G$ {\it even}  if $|V(C)|$ is even.

We call $G$ and $H$ the {\it left} and {\it  
 right factors} of the lexicographic product
 $G \circ H$, respectively.  
 The lexicographic product is also known as graph substitution, since $G \circ 
H$ is obtained from $G$ by substituting a copy $H_u$ of $H$ for  every 
vertex $u$ of $G$ and then joining all vertices of $H_u$ with all vertices of 
$H_v$ if $uv \in E(G)$\footnote{ In some literature, lexicographic product  $G\circ H$ is a special case of the blow-up of $G$, where each vertex of $G$ is blown up by  the same subgraph $H$.}.

 All graphs considered here 
are simple. A drawing is \textit{good} if it satisfies the following three conditions:
\begin{itemize}
  \item No edge crosses itself.
  \item No two edges cross more than once.
  \item No two  edges incident with a common vertex cross.
\end{itemize}

Let $G$ be a 1-planar graph with a 1-planar drawing $D$. The rotation $\Rot_D\left(u\right)$ of a vertex $u$ in $D$ as the cyclic permutation that records the (cyclic) counter-clockwise order in which the edges leave $u$. 
The rotation system naturally defines two \emph{successive} edges incident with $u$ if they appear consecutively in $\Rot_D\left(u\right)$. An edge $e$ of $G$ is 
called {\it crossed} in $D$ if it crosses with any other edge 
in $D$,  and is  {\it uncrossed } in $D$ 
otherwise. A cycle $C$ of $G$ is {\it uncrossed}  in $D$ if 
each edge on $C$ is uncrossed in $D$, and is {\it crossed } in $D$. The {\it planar skeleton} $\mathcal{S}(D)$ of $D$ is the 
subgraph of $G$ by removing all crossed edges of $D$. Let $H$ be a subgraph of 
$G$. The subdrawing $D|H$ of $H$ induced by $D$ is called a {\it restricted 
drawing} of $D$. Let $\mathcal{P}$ be a plane. Similar to planar drawings, the 
1-planar drawing $D$ also defines the {\it faces}, which are the connected 
parts of $\mathcal{P}\setminus D$. Each face $f$ contains on its boundary a 
number of 
vertices and crosseds of $D$; these are called the {\it corners} of $f$. A 
face is {\it uncrossed} if all incident corners are vertices, and is {\it 
crossed} otherwise. A {\it $k$-face} is a face  whose boundary walk has a 
length of exactly $k$. Let $L$ be a closed curve in  $\mathcal{P}$ 
that does not cross itself. Thus,  $L$  separates $\mathcal{P}$  into two open 
regions, the bounded one (i.e., the 
interior of $L$) and the unbounded one (i.e., the exterior of $L$). We denote 
by $L_{int}$ and $L_{out}$ the interior and exterior of $L$, respectively.

The structure of the remaining sections is as follows. In the next section, we first introduce auxiliary lemmas, including the 1-planarity of lexicographic products and  elementary properties on optimal 1-planar graphs. In Section \ref{sec3} and \ref{sec4}, we characterize the reducibility of optimal 1-planar graphs, depending on the order of the left and right factors. In the proof, we focus on the drawing structures of two types of small subgraphs: $C_4$ and $K_{3,3}$, along with their surrounding structures.

\section{Preliminaries}

First, we introduce two lemmas regarding the properties of the lexicographic product of graphs. Lemma \ref{lex_edges}(i) can be found in  \cite[Chapter 5, page 58]{hammack2011}, and Lemma \ref{lex_edges}(ii) follows directly from the definition of lexicographic product.

\begin{lemma}[\cite{hammack2011}]\label{lex_edges}
Let $G$ and $H$ be graphs. Then the following statements hold.
\begin{itemize}
  \item[(i)] $G\circ H$ is connected if and only if $G$ is connected, and 
  \item[(ii)] $|E(G\circ H)|=|V(G)|\times|E(H)| + 
 |V(H)|^2\times|E(G)|.$
\end{itemize}  
\end{lemma}

Since optimal 1-planar graphs are connected (in fact, 4-connected \cite{suzuki2010}), the following lemma naturally follows from Lemma \ref{lex_edges}(i).

\begin{lemma}\label{lfactor_con}
 Let $G$ be an optimal 1-planar graph. If $G=G_1\circ G_2$, then $G_1$ is connected.
\end{lemma}

The following Lemma \ref{oblex1} directly follows from the definition of lexicographic products.

\begin{lemma}\label{oblex1}
Let $G$ and $H$ be graphs. Let $V(G)=\{g_1,g_2,\dots, g_s\}$ and $V(H)=\{h_1,h_2,\dots, h_l\}$.
If  graph $G'=G\circ H$, then the following 
statements hold.
\begin{itemize}
\item [(i)] For $1\le i\le j\le s$, $G'[\{(g_i,h_1),(g_i,h_2),\dots, 
(g_i,h_l)\}]\cong G'[\{(g_j,h_1),(g_j,h_2),\dots, \\(g_j,h_l)\}]\cong H$, and
\item [(ii)] 
if a vertex $u$ in $G'$ is adjacent to $(g_k,h_i)$ where $1\le k \le s$ and 
$1\le i \le l$,
then $u$ is adjacent to each vertex in 
$\{(g_k,h_1),(g_k,h_2),\dots, (g_k,h_i),\dots, (g_k,h_{l})\}$.

\end{itemize} 
\end{lemma}
%\begin{proof}
%For (i),  by the definition of lexicographic products,  
%$G'[(g_i,h_1),(g_i,h_2),\cdots, (g_i,h_l)]$ and $G'[(g_j,h_1),(g_j,h_2),\cdots, 
%(g_j,h_l)]$ are copies of $H$  corresponding to $g_i$ and $g_j$, 
%respectively. Clearly, they are both isomorphic to $H$.
%
%For (ii), we first note that the vertex set $\{(g_k,h_1),(g_k,h_2),\cdots, 
%(g_k,h_i),\cdots, (g_k,h_{l})\}$ in $V(G')$ corresponds to the vertex 
%$g_k$ of $G$ ($g_k$ is substituted by a copy of $H$ in $V(G')$).
% Clearly, by the definition of lexicographic products, $u$ is adjacent to all 
% vertices in $\{(g_k,h_1),(g_k,h_2),\cdots, (g_k,h_i),\cdots, (g_k,h_{l})\}$, 
% since  $u$ is adjacent to $(g_k,h_i)$ in $G'$.
%\end{proof}

We present some lemmas on the 1-planarity of lexicographic products of graphs from \cite{bucko2015} and \cite{matsumoto2021}.
\begin{lemma}[\cite{bucko2015}]\label{G1_3G2_4}
Let $G$ be a connected graph with $|V(G)|\ge 3$, and let $H$ be a graph 
with $|V(H)|\ge 4$. Then $G\circ H$ is not $1$-planar.
\end{lemma}

\begin{lemma}[\cite{bucko2015}]\label{K2G2_5}
Let $H$ be a graph and let $G=K_2$. Then the following  statements hold.

\begin{itemize}
  \item[(i)]  If $|V(H)|\ge 5$,  then $G\circ H$ 
is not $1$-planar, and
  \item [(ii)] if $|V(H)|\le 4$, then $H$ is a subgraph of $C_4$ or $H$ is a 
subgraph of $C_3$. 

\end{itemize}
\end{lemma}

%\begin{lemma}[\cite{bucko2015}]\label{K2G2_4}
%Let $G=K_2$ and let $H$ be a graph with at most $4$ vertices. Then 
%$G\circ H$ is $1$-planar if and only if
%\end{lemma}

\begin{lemma}[\cite{bucko2015,matsumoto2021}]\label{lexcactus}
Let $G$ be a graph. Then $G\circ K_2$ is $1$-planar if and only if $G$ is a 
cactus.
\end{lemma}

The following result on 1-planar graphs is fundamental and has been independently obtained by many researchers.

\begin{lemma}[\cite{Bodendiek1983,suzuki2010}]\label{maxedges}
Let $G$ be a $1$-planar graph of order $n$.  Then \[
|E(G)| \le 
\begin{cases} 
\binom{n}{2} & \text{if } n \leq 6; \\
4n - 9 & \text{if } n = 7 \text{ or } 9; \\
4n - 8 & \text{otherwise} .
\end{cases}
\]
\end{lemma}

Two  drawings of a graph are {\it isomorphic} if there is a 
homeomorphism of the sphere that maps one drawing to the other; otherwise, the drawings are \textit{non-isomorphic}.\footnote{In fact, the sphere and the plane differ topologically. However, a graph is 1-planar if and only if it can be embedded on the sphere such that each edge crosses at most once, which can be achieved through stereographic projection. The only difference is that a drawing $D$ in the sphere ensures that each face of $D$ is a finite face. This definition of isomorphism for graph drawings is given on the sphere rather than the plane to avoid concerns about  the external infinite face of a 1-planar drawing.} In this following, we give  some properties of optimal 
1-planar graphs. A  plane graph is  a {\it quadrangulation} if every face is a 4-face.
References \cite{Bodendiek1983,suzuki2010} gave a  relationship 
between optimal 1-planar graphs and 3-connected quadrangulations.

\begin{lemma}[\cite{Bodendiek1983},\cite{suzuki2010}] \label{optimal_unique}
Let $G$ be an optimal 1-plane graph. Then $D$  is obtained by inserting a pair of crossed edges to each 4-face of a $3$-connected 
quadrangulation. Moreover, the 1-planar drawings of $G$ are unique (up to isomorphism).
\end{lemma}

%\begin{lemma} 
%[\cite{Bodendiek1983},\cite{suzuki2010}]\label{op3-connected}
%
%\end{lemma}

%By Lemmas \ref{optimal_unique} and \ref{op3-connected},
%optimal 1-planar graphs and 3-connected quadrangulations have the one-to-one 
%correspondence.

Lemma \ref{optimal_unique} provides powerful structural information on optimal 1-planar graphs, which directly lead to the following Lemmas \ref{alternate_edges}, \ref{2cell}, and \ref{clean_oddcycle}.

\begin{lemma}\label{alternate_edges}
Let $G$ be an optimal 1-planar graph with a 1-planar drawing $D$. Then edges incident with a vertex $v\in V(G)$ appear successively as crossed and uncrossed edges in $\Rot_D(v)$.
\end{lemma}

%\begin{proof}
%Without loss of generality, we assume that the neighbors of $v$ in clockwise 
%order in $D$ are $v_0,v_1, \cdots, v_{k-1}$. By Lemma  
%\ref{op3-connected}, only one of $vv_i$ and $vv_{i+1}$ belongs to
%$E(\mathcal{S}(D))$ where indices are taken modulo $k$, as desired. 
%\end{proof}

\begin{lemma} \label{2cell}
Let $G$ be an optimal 1-planar graph with a 1-planar drawing $D$. Let $v_0v_2$ and $v_1v_3$ be edges of $G$, which cross each other in  point $c$ in $D$. Then the 
 following two statements hold.
 \begin{itemize}
 \item[(i)] The induced subgraph $H:=G[\{v_0,v_1,v_2,v_3\}] \cong K_4$, and any edge in $E(H)\setminus \{v_0v_2,v_1v_3\}$ is uncrossed in $D$, and
  \item[(ii)] $cv_{i}v_{i+1}$  bounds a crossed face of $D$ with indices taken 
  modulo $4$.
 \end{itemize}
 \end{lemma}
 
% \begin{proof}
%For (i), it is clear to see that  $v_0v_1v_2v_3v_0$ is a 4 cycle in the 
%planar skeleton $\mathcal{S}(D)$ by Lemma \ref{op3-connected}. So (i) holds.
%
%For (ii), by Lemma \ref{op3-connected}, $v_0v_1v_2v_3v_0$  
%bounds a 4-face in $\mathcal{S}(D)$ that only contains the pair of crossed 
%edges $v_0v_2$ and $v_1v_3$. Hence, the region bounded by $v_iv_{i+1}$, $cv_i$ 
%and $cv_{i+1}$ where $1\le i \le 4$ cannot contain any vertex. So (ii) 
%holds.
% \end{proof}

\begin{lemma}[\cite{suzuki2017}] \label{clean_oddcycle}
Let $G$ be an optimal 1-planar graph with a 1-planar drawing $D$.  If $C$ is an uncrossed cycle in $D$. Then $C$ is even.
\end{lemma}

%\begin{proof}
%By Lemma \ref{op3-connected}, the planar skeleton $\mathcal{S}(D)$ is a 
%quadrangulation. Clearly, any cycle of $\mathcal{S}(D)$ is even. Since  $C$ is 
%uncrossed in $D$, $C$ is a cycle in $\mathcal{S}(D)$, and thus $C$ 
%is an even cycle of $G$, as desired.
%\end{proof}

%Unlike Maximal Planar Graphs, Optimal 1-Planar Graphs Do Not Exhibit Heredity.
The following lemma reveals that optimal 1-planar graphs do not exhibit hereditary properties with respect to optimality \footnote{A hereditary property typically refers to a property of a graph that is also preserved in its (induced) subgraphs.}, which will be used in the final step of the proof of Proposition \ref{pro1}.

\begin{lemma}\label{lem:opnonhereditary}
Any properly subgraph of an optimal 1-planar graph  $G$ is not optimal 1-planar.
\end{lemma}
\begin{proof}
Let $H$ be a properly subgraph of $G$. Suppose $H$ is optimal 1-planar with a 1-planar drawing $D'$. Clearly, since $H$ is optimal, $G[V(H)]=H$. 
Since $D'$ is unique, then $D|D'$ is fixed. by Lemma \ref{optimal_unique}, every face of $D$ is a crossed 3-face. Since $G$ is 4-connected \cite{suzuki2010}, any 3-face of $D'$ also bounds a 3-face of $D$, that is, $D$ has no vertices from $G \setminus H$.  This implies that $H \cong G$, a  contradiction.
\end{proof}

The following lemma which gives two distinct 1-planar drawings of $K_{3,3}$, particularly the first one, will be used in Section \ref{sec4}.

\begin{lemma}[\cite{Angelini2020}] \label{Two_D_K33}
There are exactly two non-isomorphic $1$-planar drawings  of 
$K_{3,3}$, as shown in Fig. \ref{K_33drawing}.
\end{lemma}

\begin{figure}[H]
\centering
\begin{tikzpicture}[scale=0.45, bezier bounding box]
	\begin{pgfonlayer}{nodelayer}
		\node [style=blacknode] (15) at (2.975, 1) {};
		\node [style=blacknode] (16) at (10.575, 1) {};
		\node [style=blacknode] (17) at (6.775, 7.75) {};
		\node [style=square] (18) at (8.525, 2.25) {};
		\node [style=square] (19) at (6.775, 5.25) {};
		\node [style=square] (36) at (5.35, 2.25) {};
		\node [style=blacknode] (74) at (-8.025, 8) {};
		\node [style=square] (75) at (-8.025, 1) {};
		\node [style=square] (76) at (-6.525, 4.5) {};
		\node [style=blacknode] (77) at (-1.025, 1) {};
		\node [style=square] (78) at (-1.025, 8) {};
		\node [style=blacknode] (79) at (-2.525, 4.5) {};
	
	\end{pgfonlayer}
	\begin{pgfonlayer}{edgelayer}
		\draw [style=black_thick] (19) to (15);
		\draw [style=black_thick] (19) to (16);
		\draw [style=black_thick] (17) to (18);
		\draw [style=black_thick] (15) to (18);
		\draw [style=black_thick] (16) to (18);
		\draw [style=black_thick] (19) to (17);
		\draw [style=black_thick] (17) to (36);
		\draw [style=black_thick] (36) to (15);
		\draw [style=black_thick] (36) to (16);
		\draw [style=black_thick] (74) to (75);
		\draw [style=black_thick] (74) to (76);
		\draw [style=black_thick] (75) to (77);
		\draw [style=black_thick] (75) to (79);
		\draw [style=black_thick] (79) to (76);
		\draw [style=black_thick] (76) to (77);
		\draw [style=black_thick] (74) to (78);
		\draw [style=black_thick] (79) to (78);
		\draw [style=black_thick] (77) to (78);
	\end{pgfonlayer}
\end{tikzpicture}
 \caption{Two non-isomorphic 1-planar drawings of $K_{3,3}$}
 \label{K_33drawing}
\end{figure}

%%\section{Proof of main results}\label{proofs}
%In this section, we prove Theorem  \ref{main} and two corollaries.
%Let $G$ be a optimal 1-planar graph with a 1-planar drawing $D$. 

\section{Left or right factor with order two}\label{sec3}
In this section, we shall prove that an optimal 1-planar graph, when reduced into factors where at least one of the left or right factors has order two, can only be isomorphic to $K_{2,2,2,2}$. 

Before that, we can easily verify the following simple lemma.

\begin{lemma}\label{lem:K2222}
The graph $K_{2,2,2,2}$ is optimal 1-planar and a 1-planar drawing of $K_{2,2,2,2}$ is shown in Fig. \ref{K2222_1}. Moreover  $K_{2,2,2,2}=K_2\circ C_4$ and $K_{2,2,2,2}=K_4\circ 2K_1$.
\end{lemma}

\begin{proposition}\label{pro1}
Let $G$ be an optimal 1-planar graph. If $G=G_1 \circ G_2$, then, $G \cong K_{2,2,2,2}$ if either of the following holds:
\begin{itemize}
  \item[(i)] $\left|V\left(G_1\right)\right|=2$ and $\left|V\left(G_2\right)\right| \geq 2$; 
  \item[(ii)] $\left|V\left(G_1\right)\right| \geq 3$ and $\left|V\left(G_2\right)\right|=2$.
\end{itemize}

\end{proposition}
\begin{proof}
If (i) holds, we have the following Claim. 

\begin{claim}
$G_2$ is $C_4$ or a $C_3$ with a pendant edge.
\end{claim}
\begin{proof}
 As $G$ is connected, by Lemma \ref{lfactor_con}, it follows that $G_1$ is connected, and thus $G_1\cong K_2$. By Lemma \ref{K2G2_5},  we have $|V(G_2)|\le 4$. Since   $|V(G)|\ge 8$ (by Lemma \ref{maxedges})
and $G_1\cong K_2$,  we have $|V(G_2)|\ge 4$. So 
$|V(G_2)|=4$. Thus $|V(G)|=8$. By the optimality of $G$, we have $|E(G)|=24$. 
Hence, by Lemma \ref{lex_edges}, we have $|E(G_2)|\ge 4$ . So $G_2$ is $C_4$ or a $C_3$ with a pendant edge. 
\end{proof}

Furthermore, by Lemma \ref{K2G2_5}(ii), $G_2 \cong 
C_4$, and thus $G\cong K_{2,2,2,2}$.

Now, we consider the case where condition (ii) of the proposition is satisfied. The situation becomes a bit more complicated. We shall deduce that $G$ is isomorphic to $K_{2,2,2,2}$ through a series of claims, starting from replacing a single edge from the left factor $G_1$ with a cycle of length four (note that the right factor shall be consist of two isolated vertices by  the following  first claim).  We will see that the drawing structure of the 4-cycle in $D$ and its surrounding elements play a key role.

\begin{claim}\label{2K1}
$|V(G_1)|\ge 4$ and $G_2\cong 2K_1$.
\end{claim}
\begin{proof}
By Lemma \ref{maxedges}, one has $|V(G)|\ge 8$. Furthermore, since  $|V(G_2)| = 2$, it follows that $|V(G_1)|\ge 4$.
Suppose  $G_2 \cong K_2$. By Lemma \ref{lex_edges}  we 
have
\begin{equation}\label{eqa}
|E(G)|=|V(G_1)|+4\times|E(G_1)|.
\end{equation}
On the other hand, by the optimality of $G$, we have
\begin{equation}\label{eqb}
|E(G)|=4\times(2\times|V(G_1)|)-8.
\end{equation}
By combining (\ref{eqa}) with (\ref{eqb}), it follows that  
\begin{equation}\label{eq0}
|E(G_1)|=\frac{7}{4}|V(G_1)|-2.
\end{equation}

%The following lemma is well known and  is given as an exercise in the
%textbook
%\begin{lemma}[\cite{west2001}]\label{edgescactus}
%Let $G$ be a cactus. Then $|E(G)|\le \Bigl\lfloor \frac{3(|V(G)|-1)}{2} 
%\Bigl\rfloor$.
%\end{lemma}

By Lemma \ref{lexcactus}, $G_1$ is a cactus. Then, by
a well-known result on the sizes of cactus graphs (\cite[Page 160]{west2001}),  we have 
\begin{equation}\label{eq1}
|E(G_1)|\le \Bigl\lfloor \frac{3(|V(G_1)|-1)}{2} \Bigl\rfloor.
\end{equation}
By (\ref{eq0}) and (\ref{eq1}), we have
 $|V(G_1)|\le 2$, a contradiction to $|V(G_1)| \ge 4$.
Hence,  $G_2\cong 2K_1$.
\end{proof}

 So we assume that $V(G_1):=\{u_1,u_2,\ldots,u_n\}$ where $n\ge 4$ and 
$V(G_2):=\{v_1,v_2\}$.
By Lemma \ref{lex_edges}(i), $G_1$ is connected. Hence, $E(G_1)\ne \emptyset$, and we may assume that $u_iu_j\in E(G_1)$  for 
$1\le i\ne j\le n$.  Then $u_i$ (resp. $u_j$) is substituted with two distinct
vertices $(u_i,v_1)$ and $(u_i,v_2)$ in $G$ (resp. $(u_j,v_1)$ and $(u_j,v_2)$ 
in $G$). 
 For simplicity, let $u:=(u_i,v_1)$, $v:=(u_i,v_2)$, $x:=(u_j,v_1)$ 
and 
$y:=(u_j,v_2)$. Let $V(G)\setminus \{u,v,x, y\}:=\{w_1,w_2,\ldots, w_{2n-4}\}$.

%Below, we give three simple claims.
 
\begin{claim}\label{fact_3}
If a vertex $p$ of $G$ is adjacent to one vertex of $\{u,v\}$ (resp. 
of $\{x,y\}$), then $p$ is also adjacent to the other vertex of $\{u,v\}$ 
(resp. of $\{x,y\}$).
\end{claim}

\begin{proof}
It follows  from Lemma \ref{oblex1} (ii) directly. 
\end{proof}

\begin{claim}\label{fact_1}
 $uv\notin E(G)$ and  $xy\notin E(G)$. Furthermore, $G[\{u,v,x,y\}]$ is a 4-cycle $uxvyu$. 
\end{claim}

\begin{proof}
By Claim \ref{2K1} and Lemma \ref{oblex1} (i), $G[\{u,v\}]\cong G[\{x,y\}]\cong 2K_1$. By the definition of lexicographic products, $\{ux,xv,vy,yu\} \subseteq E(G)$, as desired.
\end{proof}

Now, we denote the 4-cycle $uxvyu$ by $C$, and our focus below is directed towards $C$. First, we give Claims \ref{claim:Ccross} and \ref{fact_2}, which discuss the crossing properties of $C$: $C$ will not be uncrossed, and furthermore $C$ does not self-cross.

\begin{claim}\label{claim:Ccross}
$C$ is crossed in $D$.
\end{claim}
\begin{proof}

Suppose $C$ is uncrossed. We choose a neighbor of $u$, $w_1$, such that the edge $uw_1$ is successive to $ux$ in $\Rot_D(u)$ in the clockwise direction. By Lemma 
\ref{alternate_edges}, $uw_1$ is crossed in $D$, and thus $w_1\ne y$ since 
$C$ is uncrossed. By Claim \ref{fact_1}, $w_1\ne v$. So 
$w_1\in C_{int}$, shown in Fig \ref{cleancycle} (a).

 We may assume that $uw_1$ is crossed by $w_2w_3$. Notice that $w_2\ne v$ and $w_3\ne v$, otherwise, $uv\in E(G)$ by 
Lemma \ref{2cell} (i), a contradiction to Claim \ref{fact_1}.  Furthermore, 
we discuss the following three cases, each of which leads to contradictions.

\begin{case}
$w_i\cap \{x,y\}=\emptyset$ for $i=2,3$.  \textup{
Denote  by  $c$ the crossing created by $uw_1$ and $w_2w_3$. By Lemma 
\ref{2cell} (i), $uw_2$ and $ww_3$ are  uncrossed in $D$. By the choice of $uw_1$ (where $uw_1$ and $ux$ are successive edges)  the drawing shown in Fig. \ref{cleancycle} (b) will not occur. Thus $uw_2$ must be drawn shown in  Fig 
\ref{cleancycle} (c), and now $(uw_2cu)_{out}$ (or $(uw_3cu)_{out}$) is not
a crossed 3-face of $D$, a contradiction to Lemma \ref{2cell} (ii). }
\end{case}

\begin{figure}
\centering
\begin{tikzpicture}[scale=0.5]
	\begin{pgfonlayer}{nodelayer}
		\node [style=blacknode] (28) at (0.025, -3) {};
		\node [style=square] (29) at (5.025, -3) {};
		\node [style=square] (30) at (0.025, -8) {};
		\node [style=blacknode] (31) at (5.025, -8) {};
		\node [style=rednode] (32) at (3.525, -5) {};
		\node [style=none] (33) at (-0.6, -3) {$u$};
		\node [style=none] (34) at (5.6, -3) {$x$};
		\node [style=none] (35) at (-0.6, -8) {$y$};
		\node [style=none] (36) at (5.6, -8.025) {$v$};
		\node [style=rednode] (37) at (1.525, -5) {};
		\node [style=none] (38) at (4.5, -5.025) {$w_1$};
		\node [style=none] (39) at (0.6, -5.05) {$w_3$};
		\node [style=blacknode] (41) at (-0.025, 5) {};
		\node [style=square] (42) at (4.975, 5) {};
		\node [style=square] (43) at (-0.025, 0) {};
		\node [style=blacknode] (44) at (4.975, 0) {};
		\node [style=rednode] (45) at (3.475, 3) {};
		\node [style=none] (46) at (-0.6, 5) {$u$};
		\node [style=none] (47) at (5.6, 5) {$x$};
		\node [style=none] (48) at (-0.6, 0) {$y$};
		\node [style=none] (49) at (5.6, -0.025) {$v$};
		\node [style=rednode] (50) at (1.475, 3) {};
		\node [style=none] (51) at (4.4, 3.2) {$w_1$};
		\node [style=none] (52) at (0.925, 2.5) {$w_3$};
		
		\node [style=rednode] (54) at (3, 4) {};
		\node [style=none] (68) at (3.3, 4.5) {$w_2$};

		\node [style=blacknode] (70) at (-8.025, -3) {};
		\node [style=square] (71) at (-3.025, -3) {};
		\node [style=square] (72) at (-8.025, -8) {};
		\node [style=blacknode] (73) at (-3.025, -8) {};
		\node [style=rednode] (74) at (-4.525, -5) {};
		\node [style=none] (75) at (-8.7, -3) {$u$};
		\node [style=none] (76) at (-2.4, -3) {$x$};
		\node [style=none] (77) at (-8.7, -8) {$y$};
		\node [style=none] (78) at (-2.4, -8.025) {$v$};
		\node [style=rednode] (79) at (-6.525, -5) {};
		\node [style=none] (80) at (-4.4, -5.7) {$w_1$};
		\node [style=none] (81) at (-7.2, -5.05) {$w_3$};

		\node [style=rednode] (83) at (-5, -4) {};
		\node [style=none] (84) at (-4.7, -3.5) {$w_2$};
		\node [style=none] (85) at (-6, -6.75) {};
		\node [style=none] (86) at (-5.575, -5.0) {$c$};
		\node [style=blacknode] (87) at (-8.05, 5.025) {};
		\node [style=square] (88) at (-3.05, 5.025) {};
		\node [style=square] (89) at (-8.05, 0.025) {};
		\node [style=blacknode] (90) at (-3.05, 0.025) {};
		\node [style=rednode] (91) at (-4.55, 3.025) {};
		\node [style=none] (92) at (-8.7, 5.025) {$u$};
		\node [style=none] (93) at (-2.4, 5.025) {$x$};
		\node [style=none] (94) at (-8.7, 0.025) {$y$};
		\node [style=none] (95) at (-2.4, 0) {$v$};
		\node [style=none] (97) at (-3.6, 3.025) {$w_1$};
		\node at (-5.55, -1) {(a)};
		\node at (2.45, -1) {(b)};
		\node at (-5.55, -9) {(c)};
		\node at (2.45, -9) {(d)};
		\node [style=none] (100) at (-6.5, 3.5) {};
		\node [style=none] (101) at (-5.75, 4.5) {};
		\node [style=none] at (2.4, 3.0) {$c$};
	\end{pgfonlayer}
	\begin{pgfonlayer}{edgelayer}
		\draw [style=black_thick] (28) to (29);
		\draw [style=black_thick] (29) to (31);
		\draw [style=black_thick] (31) to (30);
		\draw [style=black_thick] (28) to (32);
		\draw (29) to (37);
		\draw [style=rededge] (32) to (30);
		\draw [style=rededge] (37) to (31);
		\draw [style=rededge] (32) to (29);
		\draw [style={rededge_thick}] (28) to (37);
		\draw [style={rededge_thick}] (37) to (30);
		\draw [style={black_thick}] (28) to (30);
		\draw [style=black_thick] (41) to (42);
		\draw [style=black_thick] (42) to (44);
		\draw [style=black_thick] (44) to (43);
		\draw [style=black_thick] (43) to (41);
		\draw [style=black_thick] (41) to (45);
		\draw [style=black_thick] (50) to (54);
		\draw [style=rededge] (41) to (54);
		\draw [style=black_thick] (70) to (71);
		\draw [style=black_thick] (71) to (73);
		\draw [style=black_thick] (73) to (72);
		\draw [style=black_thick] (72) to (70);
		\draw [style=black_thick] (70) to (74);
		\draw [style=black_thick] (79) to (83);
		\draw [style=rededge] (70)
			 to [in=-180, out=-75, looseness=0.75] (85.center)
			 to [in=-15, out=0, looseness=2.00] (83);
		\draw [style=black_thick] (87) to (88);
		\draw [style=black_thick] (88) to (90);
		\draw [style=black_thick] (90) to (89);
		\draw [style=black_thick] (89) to (87);
		\draw [style=black_thick] (87) to (91);
		\draw (101.center) to (100.center);
	\end{pgfonlayer}
\end{tikzpicture}
 \caption{The uncrossed $4$-cycle $uxvyu$ with the crossed edge $uw_1$ 
 in $D$}
 \label{cleancycle}
\end{figure}

\begin{case} $w_i\ne x$ and $w_j=y$ for $2\le i\ne j \le 3$.  \textup{ Similar to Case 1, this also leads to a contradiction. We omit the proof.}
\end{case}

\begin{case}$w_2=x$ or $w_3=x$.  
\textup{ We may assume, without loss of generality, that $w_2=x$. By Claim \ref{fact_1}, 
$w_3\ne y$, and thus $w_3\in C_{int}$; see Fig. \ref{cleancycle} 
(d). By Lemma  \ref{2cell} (i), $uw_3$ 
 and $xw_1$ are uncrossed edges in $D$. By Claim \ref{fact_3}, $\{w_1y,  
 w_3v,w_3y\}\subseteq E(G)$. As $C$ is uncrossed in $D$, 
 $w_1y$ must  cross $w_3v$, and thus  $w_3y$ is uncrossed  in $D$ by 
 Lemma  \ref{2cell} (i). Now the  3-cycle  $uyw_3u$ is uncrossed in 
 $D$,  a contradiction to Lemma \ref{clean_oddcycle}.}
\end{case}

Thus the claim holds.
\end{proof}

\begin{claim}\label{fact_2}
Any two edges of $C$ do not cross each other in $D$.
\end{claim}
\begin{proof} Since $D$ is
good,  it is only possible for two non-successive edges of $C$ to cross each other. Then by Lemma \ref{2cell} (i), $G[u,v,x,y]\cong
K_4$, and so $uv\in E(G)$, a contradiction to Claim \ref{fact_1}. 
\end{proof}

By Claims \ref{claim:Ccross} and \ref{fact_2},   we assume that edge $ux$ (on $C$) is crossed by some edge, namely $w_1w_2$. 
Furthermore, we claim that neither $w_1$ nor $w_2$ belongs to $C$.

\begin{claim}\label{calim:notliesC}
 $w_i\notin V(C)$ for $i=1,2$.
\end{claim} 
\begin{proof}
 As $D$ is good, $w_i\ne u$ and $w_i\ne x$ for $i=1,2$. If 
some end-vertex of $w_1w_2$, says $w_1$, is $v$, then $w_2\ne y$, otherwise $C$ 
is self-crossed in $D$, a contradiction to Claim \ref{fact_2}. Thus, $w_2\notin V(C)$, but now $uv \in E(G)$ by Lemma \ref{2cell} (i), since $v(=w_1)w_2$ crosses $ux$; this 
contradicts to Claim \ref{fact_1}. By symmetry of $w_1$ and $w_2$, we also 
have $w_2\ne v$. Similarly, $w_i\ne y$ for $i=1,2$. Thus, the claim holds.
\end{proof}

By Claims \ref{fact_2} and \ref{calim:notliesC}, we may assume that $w_1 \in C_{out}$ and $w_2\in C_{int}$, respectively. 
Furthermore, we shall prove that $C$ has exactly two crossed edges. In addition, in $C$, $vy$, which is not adjacent to $ux$, is the unique crossed edge other than $ux$s.  Before that, we observe the existence of the following subdrawing in $D$, which is an expansion of the current drawing of $C$.

\begin{claim}\label{DF}
$D$ contains a subdrawing $F$, as shown in
Fig. \ref{a_crossededge} (i).
\end{claim}
\begin{proof}
By Lemma \ref{2cell} (i), $w_1u$, $w_1x$, $w_2u$ and $w_2x$ are uncrossed 
edges in $D$. By Claim \ref{fact_3}, $\{w_1y,w_1v,w_2v,w_2y \}\subseteq E(G)$.  Since $D$ is good, $w_2v$ cannot cross $vx$ or $vy$, and 
$w_2v$ also cannot cross $uy$, otherwise $uv \in E(G)$ by  Lemma 
\ref{2cell} (i), which contradicts Claim \ref{fact_1}. So $w_2v$ must be within
$C_{int}$. Similarly, $w_2y$ is also  within  $C_{int}$. Now observe that neither $w_1y$ nor $w_1v$ can be crossed by any edge among $uy$, $vx$, and $vy$. Therefore, $w_1y$ and $w_1v$  are  within $(w_1uyvxw_1)_{out}$. Thus, $D$ contains the subdrawing $F$, as shown in Fig. \ref{a_crossededge} (i).
\end{proof}

\begin{figure}[H]
\centering
\begin{tikzpicture}[scale=0.4]
		\begin{pgfonlayer}{nodelayer}
			\node [style=blacknode] (26) at (-1.975, 4) {};
			\node [style=square] (27) at (3.025, 4) {};
			\node [style=square] (28) at (-1.975, -1) {};
			\node [style=blacknode] (29) at (3, -1) {};
			\node [style=bluenode] (30) at (0.525, 2) {};
			\node [style=none] (31) at (-2.5, 4) {$u$};
			\node [style=none] (32) at (3.6, 3.975) {$x$};
			\node [style=none] (33) at (-2.2, -1.7) {$y$};
			\node [style=none] (34) at (3, -1.7) {$v$};
			\node [style=bluenode] (35) at (0.55, 6) {};
			\node [style=none] (36) at (-0.3, 1.975) {$w_2$};
			\node [style=none] (37) at (0.5, 6.7) {$w_1$};
		\end{pgfonlayer}
		\begin{pgfonlayer}{edgelayer}
			\draw [style=black_thick] (26) to (27);
			\draw [style=black_thick] (27) to (29);
			\draw [style=black_thick] (29) to (28);
			\draw [style=black_thick] (28) to (26);
			\draw [style=black_thick] (35) to (30);
			\draw [style=black_thick] (35) to (26);
			\draw [style=black_thick] (26) to (30);
			\draw [style=black_thick] (35) to (27);
			\draw [style=black_thick] (27) to (30);
			\draw [style=black_thick] (30) to (28);
			\draw [style=black_thick] (30) to (29);
			\draw [style=black_thick, in=135, out=180, looseness=1.50] (35) to (28);
			\draw [style=black_thick, in=45, out=0, looseness=1.50] (35) to (29);
		\end{pgfonlayer}
	\end{tikzpicture}
\quad
\begin{tikzpicture}[scale=0.4]
\begin{pgfonlayer}{nodelayer}
		\node [style=blacknode] (65) at (-3.975, 2) {};
		\node [style=square] (66) at (1.025, 2) {};
		\node [style=square] (67) at (-3.975, -3) {};
		\node [style=blacknode] (68) at (1, -3) {};
		\node [style=bluenode] (69) at (-1.475, 0) {};
		\node [style=none] (70) at (-4.5, 2) {$u$};
		\node [style=none] (71) at (1.7, 1.975) {$x$};
		\node [style=none] (72) at (-3.975, -3.7) {$y$};
		\node [style=none] (73) at (0.975, -3.7) {$v$};
		\node [style=bluenode] (74) at (-1.5, 4) {};
		\node [style=none] (75) at (-0.7, -0.075) {$w_2$};
		\node [style=none] (76) at (-1.475, 4.7) {$w_1$};
		\node [style=rednode] (78) at (-1.5, -1.875) {};
		\node [style=none] (94) at (-1.5, -2.5) {$w_4$};
	\end{pgfonlayer}
	\begin{pgfonlayer}{edgelayer}
		\draw [style=black_thick] (65) to (66);
		\draw [style=black_thick] (66) to (68);
		\draw [style=black_thick] (67) to (65);
		\draw [style=black_thick] (74) to (69);
		\draw [style=black_thick,  in=135, out=180, looseness=1.50] (74) to (67);
		\draw [style=black_thick, in=45, out=0,  looseness=1.50] (74) to (68);
		\draw [style=black_thick] (74) to (65);
		\draw [style=black_thick] (65) to (69);
		\draw [style=black_thick] (74) to (66);
		\draw [style=black_thick] (66) to (69);
		\draw [style=black_thick] (69) to (67);
		\draw [style=black_thick] (69) to (68);
		\draw [style=rededge] (65) to (78);
		\draw [style=rededge] (78) to (66);
		\draw [style={rededge_thick}] (78) to (67);
		\draw [style={rededge_thick}] (78) to (68);
		\draw [style={blackedge_thick}] (67) to (68);
	\end{pgfonlayer}
\end{tikzpicture}

\caption{(i) The subdrawing $F$; (ii) the edge $w_2y$ is crossed by $uw_4$ }
\label{a_crossededge}
\end{figure}

\begin{claim}\label{ccl}
$C$ has at least two crossed edges  in $D$. 
\end{claim}
\begin{proof}
Suppose $ux$ is a unique crossed edge.  By Claim \ref{DF},  
$D$ contains  $F$ as a subdrawing.  We claim that $w_2y$ 
is crossed by some edge, say $w_3w_4$; otherwise,  the 3-cycle $uw_2yu$ is  uncrossed 
 in $D$, a 
contradiction to Lemma \ref{clean_oddcycle}.  Furthermore, we assert that 
 $w_3=u$ or $w_4=u$. If not, then $w_3$ or $w_4$, say $w_3$, 
is within  $(uw_2yu)_{int}$. By Lemma \ref{2cell} (i), $w_3y\in E(G)$, and thus 
$w_3x \in E(G)$. But $w_3x$ crosses at least two times; this  
contradicts that $D$ is 1-planar.  Thus, $w_3=u$ or $w_4=u$. Without loss 
generality,  assume  that $w_3=u$. By Claim \ref{fact_3}, $w_4\ne v$, and 
thus $w_4\in (w_2yvw_2)_{int}$; see  Fig. \ref{a_crossededge}.  
Furthermore, $w_4y$ is  uncrossed in $D$ by Lemma \ref{2cell} (i), 
and thus $w_4x\in E(G)$ by Claim \ref{fact_3}. Now  $w_4x$ 
must cross $w_2v$, and hence $w_4v$ is uncrossed in $D$ by Lemma 
\ref{2cell} (i). Note that $vy$ is also uncrossed in $D$. So $yvw_4y$ is 
uncrossed in $D$, a contradiction to Lemma 
\ref{clean_oddcycle}. Therefore, $C$ has at least two crossed edges in $D$.
\end{proof}

\begin{claim}\label{cc2}
$vx$ and $uy$  are 
uncrossed in $D$.
\end{claim}
\begin{proof}
Note that $D$ contains the subdrawing $F$. As $vx$ and 
$uy$ are symmetric in $F$, we need 
only consider $vx$. We assume that $vx$ is crossed by some edge, namely 
$w_3w_4$; see Fig. \ref{twoad_crossedge}.  Observe that if $w_i\in V(F)$ 
for $i=3,4$, then they can  be $w_1$ or $w_2$ only. So it
is sufficient to consider three cases.

\begin{case}
 $w_3\notin V(F)$ and $w_4\notin V(F)$. \textup{
Without loss of generality, we may assume that $w_3\in
(xvw_2x)_{int}$ and $w_4 \in (xvw_2x)_{out}$; see Fig 
\ref{twoad_crossedge} (a).  Since 
$w_3w_4$ crosses $vx$, by Lemma \ref{2cell} (i) $w_3v\in E(G)$, and 
thus $w_3u\in E(G)$ by Claim \ref{fact_3}. But now $u_3u$ will be 
crossed at least twice, a contradiction.}
\end{case}

\begin{case}
 $w_3\notin V(F)$ and   $w_4=w_1$.  \textup{  See Fig \ref{twoad_crossedge} (b). Similar to Case 1, $w_3u$ will be crossed at least twice,  a 
contradiction.}
\end{case}

\begin{case}
 $w_3=w_2$. \textup{ In this case, $w_4\notin V(F)$ since $G$ is simple.
Since $D$ is 1-planar, $w_4\in (xvw_1x)_{int}$, as shown in Fig 
\ref{twoad_crossedge} (c). Since $w_2~(=w_3)w_4$ crosses $vx$, $w_4v\in E(G)$ by Lemma \ref{2cell} (i), and thus by 
Claim \ref{fact_3}, $w_4u\in E(G)$ . But now $u_4u$ is crossed  (by other edges) at least twice, a 
contradiction.}
\end{case}
Thus, $xv$ is uncrossed in $D$. Similarly, $uy$ is also uncrossed 
in $D$.  So the claim follows.
\end{proof}

\begin{figure}
\centering
\begin{tikzpicture}[scale=0.45]
	\begin{pgfonlayer}{nodelayer}
		\node [style=rednode] (26) at (-9.5, 1) {};
		\node [style=rednode] (27) at (-8, 1) {};
		\node [style=none] (28) at (-11, -4) {(a)};
		\node [style=none] (30) at (-7.9, 1.5) {$w_4$};
		\node [style=none] (31) at (-9.5, 1.5) {$w_3$};
		\node [style=blacknode] (47) at (-13.525, 3) {};
		\node [style=square] (48) at (-8.525, 3) {};
		\node [style=square] (49) at (-13.525, -2) {};
		\node [style=blacknode] (50) at (-8.55, -2) {};
		\node [style=bluenode] (51) at (-11.025, 1) {};
		\node [style=none] (52) at (-14.3, 3) {$u$};
		\node [style=none] (53) at (-7.8, 2.975) {$x$};
		\node [style=none] (54) at (-13.525, -2.8) {$y$};
		\node [style=none] (55) at (-8.575, -2.8) {$v$};
		\node [style=bluenode] (56) at (-11.05, 5) {};
		\node [style=none] (57) at (-12, 0.925) {$w_2$};
		\node [style=none] (58) at (-11.025, 5.925) {$w_1$};
		\node [style=none] (76) at (11.45, -4) {(c)};
		\node [style=none] (77) at (14.2, 1.75) {$w_4$};
		\node [style=blacknode] (78) at (8.425, 3) {};
		\node [style=square] (79) at (13.425, 3) {};
		\node [style=square] (80) at (8.425, -2) {};
		\node [style=blacknode] (81) at (13.4, -2) {};
		\node [style=bluenode] (82) at (10.925, 1) {};
		\node [style=none] (83) at (7.75, 3) {$u$};
		\node [style=none] (84) at (14.4, 2.975) {$x$};
		\node [style=none] (85) at (8.425, -3.05) {$y$};
		\node [style=none] (86) at (13.375, -3.05) {$v$};
		\node [style=bluenode] (87) at (10.9, 5) {};
		\node [style=none] (88) at (9.5, 0.925) {$w_2$};
		\node [style=none] (89) at (10.925, 5.925) {$w_1$};
		\node [style=rednode] (90) at (14.2, 1) {};
		\node [style=rednode] (91) at (1.7, 1) {};
		\node [style=none] (92) at (0.475, -4) {(b)};
		\node [style=none] (93) at (1, 1) {$w_3$};
		\node [style=blacknode] (94) at (-2.55, 3) {};
		\node [style=square] (95) at (2.45, 3) {};
		\node [style=square] (96) at (-2.55, -2) {};
		\node [style=blacknode] (97) at (2.425, -2) {};
		\node [style=bluenode] (98) at (-0.05, 1) {};
		\node [style=none] (99) at (-3.3, 3) {$u$};
		\node [style=none] (100) at (3.175, 2.975) {$x$};
		\node [style=none] (101) at (-2.55, -2.8) {$y$};
		\node [style=none] (102) at (2.4, -2.775) {$v$};
		\node [style=bluenode] (103) at (-0.075, 5) {};
		\node [style=none] (104) at (-1.2, 1.175) {$w_2$};
		\node [style=none] (105) at (-0.05, 5.925) {$w_1$};
	\end{pgfonlayer}
	\begin{pgfonlayer}{edgelayer}
		\draw [style=rededge] (26) to (27);
		\draw [style={black_thick}] (47) to (48);
		\draw [style={black_thick}] (48) to (50);
		\draw [style={black_thick}] (49) to (47);
		\draw [style={black_thick}] (56) to (51);
		\draw [style={black_thick}, in=135, out=180, looseness=1.50] (56) to (49);
		\draw [style={black_thick}, in=45, out=0, looseness=1.50] (56) to (50);
		\draw [style={black_thick}] (56) to (47);
		\draw [style={black_thick}] (47) to (51);
		\draw [style={black_thick}] (56) to (48);
		\draw [style={black_thick}] (48) to (51);
		\draw [style={black_thick}] (51) to (49);
		\draw [style={black_thick}] (51) to (50);
		\draw [style={black_thick}] (49) to (50);
		\draw [style={black_thick}] (78) to (79);
		\draw [style={black_thick}] (79) to (81);
		\draw [style={black_thick}] (80) to (78);
		\draw [style={black_thick}] (87) to (82);
		\draw [style={black_thick}, in=135, out=180, looseness=1.50] (87) to (80);
		\draw [style={black_thick}, in=45, out=0, looseness=1.50] (87) to (81);
		\draw [style={black_thick}] (87) to (78);
		\draw [style={black_thick}] (78) to (82);
		\draw [style={black_thick}] (87) to (79);
		\draw [style={black_thick}] (79) to (82);
		\draw [style={black_thick}] (82) to (80);
		\draw [style={black_thick}] (82) to (81);
		\draw [style={black_thick}] (80) to (81);
		\draw [style=rededge] (82) to (90);
		\draw [style=rededge] (90) to (81);
		\draw [style={black_thick}] (94) to (95);
		\draw [style={black_thick}] (95) to (97);
		\draw [style={black_thick}] (96) to (94);
		\draw [style={black_thick}] (103) to (98);
		\draw [style={black_thick}, in=135, out=180, looseness=1.50] (103) to (96);
		\draw [style={black_thick}, in=45, out=0, looseness=1.50] (103) to (97);
		\draw [style={black_thick}] (103) to (94);
		\draw [style={black_thick}] (94) to (98);
		\draw [style={black_thick}] (103) to (95);
		\draw [style={black_thick}] (95) to (98);
		\draw [style={black_thick}] (98) to (96);
		\draw [style={black_thick}] (98) to (97);
		\draw [style={black_thick}] (96) to (97);
		\draw [style=rededge, in=15, out=-15, looseness=2.00] (103) to (91);
		\draw [style=rededge] (91) to (97);
		\draw [style=rededge] (50) to (26);
	\end{pgfonlayer}
\end{tikzpicture}

\caption{Three 1-planar drawings involving the crossed edge $xv$}
\label{twoad_crossedge}
\end{figure}

\begin{figure}
\centering
\begin{tikzpicture}[scale=0.45]
	\begin{pgfonlayer}{nodelayer}
		\node [style=square] (13) at (-4.45, 5) {};
		\node [style=blacknode] (14) at (0.55, 5) {};
		\node [style=blacknode] (15) at (-4.45, 0) {};
		\node [style=square] (16) at (0.525, 0) {};
		\node [style=bluenode] (17) at (-1.95, 2.5) {};
		\node [style=none] (18) at (-5.4, 5) {$u$};
		\node [style=none] (19) at (1.35, 5) {$x$};
		\node [style=none] (20) at (-4.45, -0.75) {$y$};
		\node [style=none] (21) at (0.525, -0.775) {$v$};
		\node [style=bluenode] (22) at (-1.975, 7) {};
		\node [style=none] (23) at (-2.79, 2.49) {$w_2$};
		\node [style=none] (24) at (-1.975, 7.9) {$w_1$};
		\node [style=rednode] (25) at (-2, -1) {};
		\node [style=rednode] (26) at (-2, 0.75) {};
		\node [style=none] (27) at (-1.15, 0.8) {$w_4$};
		\node [style=none] (28) at (-2, -1.7) {$w_3$};
	\end{pgfonlayer}
	\begin{pgfonlayer}{edgelayer}
		\draw [style={black_thick}] (13) to (14);
		\draw [style={black_thick}] (14) to (16);
		\draw [style={black_thick}] (16) to (15);
		\draw [style={black_thick}] (15) to (13);
		\draw [style={black_thick}] (22) to (17);
		\draw [style={black_thick}] (22) to (13);
		\draw [style={black_thick}] (13) to (17);
		\draw [style={black_thick}] (22) to (14);
		\draw [style={black_thick}] (14) to (17);
		\draw [style={black_thick}] (17) to (15);
		\draw [style={black_thick}] (17) to (16);
		\draw [style={black_thick}, in=135, out=180, looseness=1.50] (22) to (15);
		\draw [style={black_thick}, in=45, out=0, looseness=1.50] (22) to (16);
		\draw [style=rededge] (26) to (25);
		\draw [style=rededge] (26) to (15);
		\draw [style=rededge] (15) to (25);
		\draw [style=rededge] (25) to (16);
		\draw [style=rededge] (16) to (26);
		\draw [style=rededge] (26) to (13);
		\draw [style=rededge] (26) to (14);
		\draw [style=rededge, in=210, out=-165, looseness=1.75] (25) to (13);
		\draw [style=rededge, in=315, out=-15, looseness=1.75] (25) to (14);
	\end{pgfonlayer}
\end{tikzpicture}

\caption{ The 1-planar drawing $F'$}
\label{K2222}
\end{figure}

By Claims \ref{ccl} and \ref{cc2}, the following  claim follows directly. 

\begin{claim}
$vy$  is crossed by some edge, namely $w_3w_4$. 
\end{claim}  

Furthermore, we shall prove that $w_3\notin V(F)$ and  $w_4 \notin V(F)$.

\begin{claim}\label{c11}
$w_3\notin V(F)$ and  $w_4 \notin V(F)$.
\end{claim}
\begin{proof} Since 
$w_3w_4$ has been crossed by $vy$, if $w_3\in V(F)$ or $w_4 \in V(F)$, then they can  be $w_1$ or $w_2$ only.  Without loss of generality, assume that $w_3 = w_1$, and thus $w_4\ne w_2$ by simplify of $G$.  
Thus, $w_4\in (w_2vyw_2)_{int}$. By Lemma 
\ref{2cell} (i),
$w_1v$ is  uncrossed in $D$. Note 
that $w_1x$ and $uy$ are uncrossed in $D$. So the 
3-cycle $vw_1yv$ is uncrossed in $D$, 
which  contradicts Lemma \ref{clean_oddcycle}. If $w_3=w_2$, then 
similarly, there exists an uncrossed 3-cycle $w_2uyw_2$ in $D$, 
a contradiction.  By symmetry, 
$w_4\ne w_i$  for $i=1,2$. Therefore, $w_3\notin V(F)$ and  $w_4 \notin V(F)$. 
\end{proof}

By Claim \ref{c11}, we can assume that $w_3$ and $w_4$ are within $(w_2vyw_2)_{out}$ and
$(w_2yvw_2)_{int}$, respectively. Furthermore, since $D$ is 1-planar and $w_3w_4$ crosses
$vy$,  we have $w_3\in (w_1yvw_1)_{out}$.

\begin{claim}\label{c12}
 $D$ contains a 1-planar drawing $F'$, shown in Fig. \ref{K2222}. 
\end{claim}
\begin{proof}
Similar to Claim \ref{DF}, we obtain this claim.
\end{proof}

Since $w_1v$ crosses $w_3x$ in $F'$,  by Lemma \ref{2cell} (i), $w_1w_3$ is uncrossed 
in $D$. Similarly, $w_2w_4$ is also  uncrossed 
in $D$. So we obtain  graph $F''$, which obtained from $F'$ by 
adding two uncrossed edges $w_1w_3$ and $w_2w_4$.  Now $F''\cong K_{2,2,2,2}$, and $F''$ is optimal 1-planar. So by Lemma \ref{lem:opnonhereditary},  $G\cong F''$. Combining this with Lemma \ref{lem:K2222}, the proposition follows.
\end{proof}

\section{Left and right factors with order at least three}\label{sec4}
In this section, we shall prove that an optimal 1-planar graph cannot be reduced into factors where both the left and right factors have order greater than three; see Proposition \ref{pro3}. This proof approach of Proposition \ref{pro3} is somewhat similar to that of Proposition \ref{pro1}(ii), but due to the order (which is three) of the right factor. Our focus is on considering a specific 1-planar drawing of subgraph $K_{3,3}$. In fact, the proof is simpler than in Proposition \ref{pro1}(ii), as there is more direct local information available to utilize compared to the case of a cycle of length four.
\begin{proposition}\label{pro3}
 Let $G$ be an optimal 1-planar graph. If $|V(G_1)|\ge 3$ and $|V(G_2)|\ge 3$, then $G\ne  G_1\circ G_2$.
\end{proposition} 

\begin{proof}

Suppose $G=G_1\circ G_2$. If 
$|V(G_2)|\ge 4$, then by Lemma \ref{G1_3G2_4},
$G_1\circ G_2$ is not 1-planar, a contradiction. Hence, 
we only consider  $|V(G_2)|=3$.  Let $V(G_1)=\{u_1,u_2,\dots, u_k\}$ where $k\ge 3$ and $V(G_2)=\{v_1,v_2, v_3\}$.

\begin{claim}
$G$ contains a copy of $K_{3,3}$.
\end{claim}
\begin{proof}
By Lemma \ref{lfactor_con}, $G_1$ is connected. Hence, there is an edge 
$u_iu_j$ for $1\le i\ne j\le k$ in $G_1$.  Since $|V(G_2)|\ge 3$, then 
$G$ contains a copy of $K_{3,3}$.
\end{proof}

We denote one copy of $K_{3,3}$ by $T$, where two parts of $T$ are 
$X:=\{(u_i,v_1),(u_i,v_2), (u_i,v_3$)\} and $Y:=\{(u_j,v_1),(u_j,v_2), 
(u_j,v_3)\}$. For brevity,  we relabel $X:=\{u,v,w\}$ and $Y:=\{x,y,z\}$.
Then two observations below follow from  Lemma \ref{oblex1}  directly.

\begin{observation}\label{cla1}
If a vertex $p$ of $G$ is adjacent to one vertex of $X$ (resp. $Y$), then $p$ 
is adjacent to  all vertices of $X$ (resp. $Y$).
\end{observation}

\begin{observation}\label{cla2}
$G[\{x,y,z\}]\cong G[\{u,v,w\}] $.
\end{observation}

%\begin{claim}
% $D$ contains two possible restricted drawings (A1) and 
%(A2) of $D|T$ 
%\end{claim}
%\begin{proof}
%By Lemma \ref{Two_D_K33},  two possible restricted drawings (A1) and 
%(A2) of $D|T$ are shown in Fig. \ref{subK_33}. 
%\end{proof}

\begin{claim}\label{cA2}
$D|T\cong A1$.
\end{claim}
\begin{proof}

By Lemma \ref{Two_D_K33},  $D|T$  is isomorphic to A1 or A2.
If $D|T\cong A2$, by Lemma \ref{2cell} (i), $D$ has the 3-cycle $xyzx$, and it is  uncrossed in $D$, a contradiction to Lemma \ref{clean_oddcycle}. So $D|T\cong A2$.
\end{proof}

\begin{claim}\label{DH}
$D$ contains $D|H$, shown in  Fig \ref{D'} , as  a  subdrawing.  
\end{claim}
 \begin{proof}
By Claim \ref{cA2}, $T\cong A1$.
 Since $xv$ crosses $uy$ in $D|H$,  we have 
$\{uv, xy\} \subseteq E(G)$. Then $uv$ and $xy$ are  uncrossed in $D$ by Lemma 
\ref{2cell} (i). Similarly, $ux$ and $vy$ are uncrossed.
Let $H$ be the subgraph of $G$ obtained from $T$ by adding $xy$ and 
$uv$. So $D$ contains the subdrawing $D|H$  
 \end{proof}

Since $uv$ is uncrossed in
$D$, by Lemma \ref{2cell} (i), $uz$ or $vz$ is crossed in $D$.  By symmetry,
 we may assume that $uz$ is crossed by some edge $st$.
By symmetry of  $s$ and $t$, we discuss five cases, each of which leads to a contradiction regarding 1-planarity or an uncrossed 3-cycle.

\begin{case} $s\notin V(H)$ and  $t\notin V(H)$.\textup{ We may assume that $s\in(uvzu)_{int}$ and $t\in (uvzu)_{out}$, shown in Fig. \ref{K33_cont} (1). As $st$ crosses $uz$, by Lemma \ref{2cell} (i) $su$ is
uncrossed in $D$. By Observation  \ref{cla1}, $\{sw,sv \}\subseteq G$. Since $uv$ is uncrossed in $D$ and $uz$ has been crossed by $st$, $sw$ crosses $vz$, and hence $sv$ is uncrossed 
in $D$ by Lemma \ref{2cell}. Thus,  the 3-cycle$uvsu$ is
uncrossed  in $D$, a contradiction to Lemma \ref{clean_oddcycle}. }
\end{case}

\begin{case}  $s\notin V(H)$ and $t=w$. \textup{  In this case, $s\in (uvzu)_{int}$; see Fig. \ref{K33_cont} (2). Then $sz\in E(G)$ by 
Lemma \ref{2cell} (i), since $sw$ is crossed by $uz$. By  Observation  \ref{cla1}, $sx\in E(G)$. But  $sx$ is crossed twice or will make at 
least one edge on $D|H$ to be crossed at least twice, a contradiction.}
\end{case}

\begin{case}
 $s\notin V(H)$ and  $t=x$. \textup{  In this case,  $s$ lies in $(uvzu)_{int}$; see Fig. 
\ref{K33_cont} (3). Since $sx$ crosses $uz$, it follows that $sz,su \in E(G)$ by Lemma \ref{2cell} (i). By  Observation  \ref{cla1},
$\{sy,sw \}\subseteq E(G)$. As $uz$ is crossed and $uv$ is 
uncrossed in $D$, $sy$ or $sw$, say $sy$, must cross $vz$, 
and hence $sw$ will cross twice or  cause at least one 
edge on $D|H$ to be crossed at least twice, a contradiction.}
\end{case}

\begin{case}
 $s=v$ and $t\notin V(H)$. \textup{In this case,  $uz$ is crossed by $tv$, and $t\in (uxwzu)_{out}$.
 By Lemma \ref{2cell} (i), $tz \in E(G)$, $tu \in E(G)$, and $vz$ and $tu$ are uncrossed in $D$.  Furthermore, by 
 Observation  \ref{cla1}, $ty\in E(G)$.  Then $sy$ crosses $wz$ 
or $wx$. If $ty$ crosses $wz$, then  $yz$ is  uncrossed in $D$ by 
Lemma \ref{2cell} (i), and hence the 3-cycle $vyzv$ is  uncrossed 
 in $D$; see Fig.\ref{K33_cont} (4), a contradiction 
to Lemma \ref{clean_oddcycle}. If $ty$ crosses $wz$, then $tx\in E(D)$ and is 
uncrossed  by Lemma \ref{2cell} (i), as shown in Fig. \ref{K33_cont} 
(5). But now the 3-cycle $tuxt$ is uncrossed  in $D$, a
contradiction to Lemma \ref{clean_oddcycle}.} 
\end{case}

\begin{case} $s=v$ and $t=w$. \textup{In this case, since $vw$ crosses $uz$, $uw$ is an edge of $G$ and $vz$ is 
uncrossed in $D$ by Lemma \ref{2cell} (i); see Fig. \ref{K33_cont} (6). 
Now observe that $G[\{u,v,w\}]\cong C_3$, and thus $G[\{x,y,z\}]\cong C_3$ by Observation \ref{cla2}. Hence, $yz\in E(G)$ and $xy\in E(G)$. Then $xz$ must cross $wy$, and thus $yz$ is uncrossed in $D$ 
by Lemma \ref{2cell}. Recall that $vy$ is also uncrossed in $D$. Thus, 
the 3-cycle $yvzy$ is uncrossed  in $D$, a contradiction to Lemma 
\ref{clean_oddcycle}.}
\end{case}

Thus, we conclude that $G \neq G_1 \circ G_2$, and with this, the proposition is established.
\end{proof}

\begin{figure}
\centering
\begin{tikzpicture}[scale=0.35]
	\begin{pgfonlayer}{nodelayer}
		\node [style=blacknode] (0) at (2.125, -1) {};
		\node [style=blacknode] (1) at (9.725, -1) {};
		\node [style=blacknode] (2) at (5.925, 5.75) {};
		\node [style=square] (3) at (7.675, 0.25) {};
		\node [style=square] (4) at (5.925, 3.25) {};
		\node [style=square] (5) at (4.5, 0.25) {};
		\node [style=blacknode] (6) at (-9.1, 6) {};
		\node [style=square] (7) at (-9.1, -1) {};
		\node [style=square] (8) at (-7.6, 2.5) {};
		\node [style=blacknode] (9) at (-2.1, -1) {};
		\node [style=square] (10) at (-2.1, 6) {};
		\node [style=blacknode] (11) at (-3.6, 2.5) {};
		\node [style=none] (12) at (-9.075, -1.75) {$u$};
		\node [style=none] (13) at (-7.45, 3.25) {$v$};
		\node [style=none] (14) at (-2.125, 7) {$w$};
		\node [style=none] (15) at (-2.1, -1.8) {$x$};
		\node [style=none] (16) at (-3.75, 3.225) {$y$};
		\node [style=none] (17) at (-9.075, 7) {$z$};
		\node [style=none] (18) at (5.225, 0.7) {$u$};
		\node [style=none] (19) at (6.97, 0.7) {$v$};
		\node [style=none] (20) at (5.925, 2.45) {$w$};
		\node [style=none] (21) at (2.175, -1.75) {$x$};
		\node [style=none] (22) at (9.675, -1.75) {$y$};
		\node [style=none] (23) at (5.925, 6.725) {$z$};
		\node [style=none] (24) at (-5.55, -2.75) {(A1)};
		\node [style=none] (25) at (5.95, -2.75) {(A2)};
	\end{pgfonlayer}
	\begin{pgfonlayer}{edgelayer}
		\draw [style={black_thick}] (4) to (0);
		\draw [style={black_thick}] (4) to (1);
		\draw [style={black_thick}] (2) to (3);
		\draw [style={black_thick}] (0) to (3);
		\draw [style={black_thick}] (1) to (3);
		\draw [style={black_thick}] (4) to (2);
		\draw [style={black_thick}] (2) to (5);
		\draw [style={black_thick}] (5) to (0);
		\draw [style={black_thick}] (5) to (1);
		\draw [style={black_thick}] (6) to (7);
		\draw [style={black_thick}] (6) to (8);
		\draw [style={black_thick}] (7) to (9);
		\draw [style={black_thick}] (7) to (11);
		\draw [style={black_thick}] (11) to (8);
		\draw [style={black_thick}] (8) to (9);
		\draw [style={black_thick}] (6) to (10);
		\draw [style={black_thick}] (11) to (10);
		\draw [style={black_thick}] (9) to (10);
	\end{pgfonlayer}
\end{tikzpicture}

 \caption{Two possible restricted 1-planar drawings of $D|T$}
 \label{subK_33}
\end{figure}
% fig/subK33new.
 
\begin{figure}
\centering
\begin{tikzpicture}[scale=0.45]
	\begin{pgfonlayer}{nodelayer}
		\node [style=blacknode] (0) at (-5, 6) {};
		\node [style=square] (1) at (-5, -1) {};
		\node [style=square] (2) at (-3.5, 2.5) {};
		\node [style=blacknode] (3) at (2, -1) {};
		\node [style=square] (4) at (2, 6) {};
		\node [style=blacknode] (5) at (0.5, 2.5) {};
		\node [style=none] (6) at (-4.975, -2) {$u$};
		\node [style=none] (7) at (-2.95, 3.25) {$v$};
		\node [style=none] (8) at (1.975, 7) {$w$};
		\node [style=none] (9) at (2, -1.8) {$x$};
		\node [style=none] (10) at (0.1, 3.225) {$y$};
		\node [style=none] (11) at (-4.975, 7.05) {$z$};
	\end{pgfonlayer}
	\begin{pgfonlayer}{edgelayer}
		\draw [style={black_thick}] (0) to (1);
		\draw [style={black_thick}] (0) to (2);
		\draw [style={black_thick}] (1) to (3);
		\draw [style={black_thick}] (1) to (5);
		\draw [style={black_thick}] (5) to (2);
		\draw [style={black_thick}] (2) to (3);
		\draw [style={black_thick}] (0) to (4);
		\draw [style={black_thick}] (5) to (4);
		\draw [style={black_thick}] (3) to (4);
		\draw [style={black_thick}] (2) to (1);
		\draw [style={black_thick}] (5) to (3);
	\end{pgfonlayer}
\end{tikzpicture}

\caption{The subdrawing $D|H$}
\label{D'}
\end{figure}

\begin{figure}[H]
\centering
\begin{tikzpicture}[scale=0.35]
	\begin{pgfonlayer}{nodelayer}
		\node [style=blacknode] (0) at (-16, 14.025) {};
		\node [style=square] (1) at (-16, 7.025) {};
		\node [style=square] (2) at (-14.5, 10.525) {};
		\node [style=blacknode] (3) at (-9, 7.025) {};
		\node [style=square] (4) at (-9, 14.025) {};
		\node [style=blacknode] (5) at (-10.5, 10.525) {};
		\node [style=none] (6) at (-16.225, 6.25) {$u$};
		\node [style=none] (7) at (-13.85, 11.2) {$v$};
		\node [style=none] (8) at (-9.025, 15.1) {$w$};
		\node [style=none] (9) at (-9, 5.975) {$x$};
		\node [style=none] (10) at (-10.7, 11.5) {$y$};
		\node [style=none] (11) at (-15.975, 15.05) {$z$};
		\node [style=rednode] (12) at (-15.5, 10.775) {};
		\node [style=rednode] (13) at (-16.75, 10.775) {};
		\node [style=none] (14) at (-12.5, 3.525) {(1)};
		\node [style=blacknode] (15) at (-4, -0.975) {};
		\node [style=square] (16) at (-4, -7.975) {};
		\node [style=square] (17) at (-2.5, -4.475) {};
		\node [style=blacknode] (18) at (3, -7.975) {};
		\node [style=square] (19) at (3, -0.975) {};
		\node [style=blacknode] (20) at (1.5, -4.475) {};
		\node [style=none] (21) at (-3.975, -8.73) {$u$};
		\node [style=none] (22) at (-2.025, -3.475) {$v$};
		\node [style=none] (23) at (2.975, -0.22) {$w$};
		\node [style=none] (24) at (3, -8.95) {$x$};
		\node [style=none] (25) at (1.1, -3.4) {$y$};
		\node [style=none] (26) at (-3.975, -0.2) {$z$};
		\node [style=rednode] (28) at (-5.5, -4.475) {};
		\node [style=blacknode] (29) at (-16, -0.975) {};
		\node [style=square] (30) at (-16, -7.975) {};
		\node [style=square] (31) at (-14.5, -4.475) {};
		\node [style=blacknode] (32) at (-9, -7.975) {};
		\node [style=square] (33) at (-9, -0.975) {};
		\node [style=blacknode] (34) at (-10.5, -4.475) {};
		\node [style=none] (35) at (-15.975, -8.98) {$u$};
		\node [style=none] (36) at (-14, -3.75) {$v$};
		\node [style=none] (37) at (-9.025, -0.225) {$w$};
		\node [style=none] (38) at (-9, -9.025) {$x$};
		\node [style=none] (39) at (-9.65, -4.4) {$y$};
		\node [style=none] (40) at (-15.975, -0.3) {$z$};
		\node [style=rednode] (42) at (-17.5, -4.475) {};
		\node [style=blacknode] (43) at (-4, 14.025) {};
		\node [style=square] (44) at (-4, 7.025) {};
		\node [style=square] (45) at (-2.5, 10.525) {};
		\node [style=blacknode] (46) at (3, 7.025) {};
		\node [style=square] (47) at (3, 14.025) {};
		\node [style=blacknode] (48) at (1.5, 10.525) {};
		\node [style=none] (49) at (-3.975, 6.25) {$u$};
		\node [style=none] (50) at (-2.225, 11.25) {$v$};
		\node [style=none] (51) at (2.975, 15.025) {$w$};
		\node [style=none] (52) at (3, 5.975) {$x$};
		\node [style=none] (53) at (1.05, 11.35) {$y$};
		\node [style=none] (54) at (-4, 15.0) {$z$};
		\node [style=rednode] (55) at (-3.5, 10.525) {};
		\node [style=blacknode] (57) at (8, 14.025) {};
		\node [style=square] (58) at (8, 7.025) {};
		\node [style=square] (59) at (9.5, 10.525) {};
		\node [style=blacknode] (60) at (15, 7.025) {};
		\node [style=square] (61) at (15, 14.025) {};
		\node [style=blacknode] (62) at (13.5, 10.525) {};
		\node [style=none] (63) at (8.025, 6.25) {$u$};
		\node [style=none] (64) at (10.05, 11.25) {$v$};
		\node [style=none] (65) at (14.975, 15.025) {$w$};
		\node [style=none] (66) at (15, 6.225) {$x$};
		\node [style=none] (67) at (12.95, 11.25) {$y$};
		\node [style=none] (68) at (8.025, 14.8) {$z$};
		\node [style=rednode] (69) at (8.5, 10.775) {};
		\node [style=blacknode] (71) at (8, -0.975) {};
		\node [style=square] (72) at (8, -7.975) {};
		\node [style=square] (73) at (9.5, -4.475) {};
		\node [style=blacknode] (74) at (15, -7.975) {};
		\node [style=square] (75) at (15, -0.975) {};
		\node [style=blacknode] (76) at (13.5, -4.475) {};
		\node [style=none] (77) at (8.025, -8.98) {$u$};
		\node [style=none] (78) at (9.8, -4) {$v$};
		\node [style=none] (79) at (14.975, -0.225) {$w$};
		\node [style=none] (80) at (15, -8.95) {$x$};
		\node [style=none] (81) at (13.35, -3.55) {$y$};
		\node [style=none] (82) at (8.025, -0.2) {$z$};
		\node [style=none] (85) at (-12.5, -11.475) {(4)};
		\node [style=none] (86) at (-0.5, -11.475) {(5)};
		\node [style=none] (87) at (-0.5, 3.525) {(2)};
		\node [style=none] (88) at (11.5, 3.525) {(3)};
		\node [style=none] (89) at (11.5, -11.475) {(6)};
		\node [style=none] (90) at (-15.5, 11.5) {$s$};
		\node [style=none] (91) at (-17.25, 10.75) {$t$};
		\node [style=none] (92) at (-2.75, 1.025) {};
		\node [style=none] (93) at (-18.5, -4.475) {$t$};
		\node [style=none] (94) at (-6.25, -4.475) {$t$};
		\node [style=none] (95) at (-12.5, 8.225) {$c$};
		\node [style=none] (96) at (-12.525, -6.7) {$c$};
		\node [style=none] (97) at (-0.5, -6.65) {$c$};
		\node [style=none] (98) at (11.5, -6.6) {$c$};
		\node [style=none] (99) at (-0.5, 8.375) {$c$};
		\node [style=none] (100) at (11.425, 8.3) {$c$};
		\node [style=none] (102) at (8.95, 10.25) {$s$};
		\node [style=none] (103) at (-3.25, 9.775) {$s$};
	\end{pgfonlayer}
	\begin{pgfonlayer}{edgelayer}
		\draw [style={black_thick}] (0) to (1);
		\draw [style={black_thick}] (0) to (2);
		\draw [style={black_thick}] (1) to (3);
		\draw [style={black_thick}] (1) to (5);
		\draw [style={black_thick}] (5) to (2);
		\draw [style={black_thick}] (2) to (3);
		\draw [style={black_thick}] (0) to (4);
		\draw [style={black_thick}] (5) to (4);
		\draw [style={black_thick}] (3) to (4);
		\draw [style={black_thick}] (5) to (3);
		\draw (12) to (13);
		\draw [style={black_thick}] (2) to (1);
		\draw [style={rededge_thick}] (12) to (1);
		\draw [style={rededge_thick}] (12) to (2);
		\draw [style={black_thick}] (15) to (16);
		\draw [style={black_thick}] (15) to (17);
		\draw [style={black_thick}] (16) to (18);
		\draw [style={black_thick}] (16) to (20);
		\draw [style={black_thick}] (20) to (17);
		\draw [style={black_thick}] (17) to (18);
		\draw [style={black_thick}] (15) to (19);
		\draw [style={black_thick}] (20) to (19);
		\draw [style={black_thick}] (18) to (19);
		\draw [style={black_thick}] (20) to (18);
		\draw [style={black_thick}] (17) to (16);
		\draw [style={black_thick}] (29) to (30);
		\draw [style={black_thick}] (30) to (32);
		\draw [style={black_thick}] (30) to (34);
		\draw [style={black_thick}] (31) to (32);
		\draw [style={black_thick}] (29) to (33);
		\draw [style={black_thick}] (34) to (33);
		\draw [style={black_thick}] (32) to (33);
		\draw [style={black_thick}] (34) to (32);
		\draw [style={black_thick}] (31) to (30);
		\draw [style={black_thick}] (43) to (44);
		\draw [style={black_thick}] (43) to (45);
		\draw [style={black_thick}] (44) to (46);
		\draw [style={black_thick}] (44) to (48);
		\draw [style={black_thick}] (48) to (45);
		\draw [style={black_thick}] (45) to (46);
		\draw [style={black_thick}] (43) to (47);
		\draw [style={black_thick}] (48) to (47);
		\draw [style={black_thick}] (46) to (47);
		\draw [style={black_thick}] (48) to (46);
		\draw [style={black_thick}] (45) to (44);
		\draw [style={black_thick}] (57) to (58);
		\draw [style={black_thick}] (57) to (59);
		\draw [style={black_thick}] (58) to (60);
		\draw [style={black_thick}] (58) to (62);
		\draw [style={black_thick}] (62) to (59);
		\draw [style={black_thick}] (59) to (60);
		\draw [style={black_thick}] (57) to (61);
		\draw [style={black_thick}] (62) to (61);
		\draw [style={black_thick}] (60) to (61);
		\draw [style={black_thick}] (62) to (60);
		\draw [style={black_thick}] (59) to (58);
		\draw [style=rededge] (69) to (58);
		\draw [style={black_thick}] (71) to (72);
		\draw [style={black_thick}] (71) to (73);
		\draw [style={black_thick}] (72) to (74);
		\draw [style={black_thick}] (72) to (76);
		\draw [style={black_thick}] (76) to (73);
		\draw [style={black_thick}] (73) to (74);
		\draw [style={black_thick}] (71) to (75);
		\draw [style={black_thick}] (76) to (75);
		\draw [style={black_thick}] (74) to (75);
		\draw [style={black_thick}] (76) to (74);
		\draw [style={black_thick}] (73) to (72);
		\draw [style=rededge] (17) to (28);
		\draw [style=rededge] (15) to (28);
		\draw [style={rededge_thick}] (16) to (28);
		\draw [style={rededge_thick}, in=-165, out=-105, looseness=1.75] (28) to (18);
		\draw [style=rededge, in=105, out=-165] (92.center) to (28);
		\draw [style=rededge, in=45, out=15, looseness=2.25] (92.center) to (20);
		\draw [style=rededge, in=90, out=105, looseness=2.25] (42) to (34);
		\draw [style=rededge] (29) to (42);
		\draw [style=rededge] (31) to (42);
		\draw [style={rededge_thick}] (29) to (34);
		\draw [style={black_thick}] (29) to (31);
		\draw [style={black_thick}] (31) to (34);
		\draw [style=rededge, in=150, out=120, looseness=2.25] (55) to (47);
		\draw [style=rededge] (12) to (4);
		\draw [style=rededge, in=-150, out=-120, looseness=2.25] (69) to (60);
		\draw [style=rededge] (57) to (69);
		\draw [style=rededge, in=150, out=135, looseness=2.75] (73) to (75);
		\draw [style=rededge, in=-10, out=100, looseness=1.50] (74) to (71);
		\draw [style=rededge, in=-75, out=-30, looseness=2.25] (72) to (75);
		\draw [style={rededge_thick}] (76) to (71);
		\draw [style=rededge] (55) to (43);
		\draw [style=rededge, in=75, out=90, looseness=1.50] (69) to (62);
		\draw [style=rededge] (42) to (30);
	\end{pgfonlayer}
\end{tikzpicture}

				\caption{The possible six subdrawings involving a crossed edge $uz$}
\label{K33_cont}
\end{figure}
%NEW2

\section{Proof of Theorems \ref{main}, \ref{coro1} and \ref{coro0}}

\noindent {\bf  Proof of Theorem \ref{main}.} The ``only if'' part of the theorem  follows directly from Lemma \ref{lem:K2222}. By combining Proposition \ref{pro1} with Proposition \ref{pro3}, the ``if'' part of the theorem is established. \proofend

\noindent {\bf  Proof of Theorem \ref{coro1}.} Since  $n\ge 5$,  
by Theorem \ref{main}, $G\circ 2K_1$ is not optimal 1-planar. Thus, 
$|E(G\circ 2K_1)|\le 4\times 2n-9$.  By Lemma \ref{lex_edges}, $|E(G\circ 
2K_1)|=4m$. So  $4m\le 8n-9$, and thus 
$m\le \lfloor\frac{8n-9}{4}\rfloor=2n-3$, as desired.  
 \proofend

\noindent {\bf  Proof of Theorem \ref{coro0}.}
If $n\ge 10$, then $m\le 4n-8$ by Lemma \ref{maxedges}. Theorem 
\ref{main} implies  that any optimal $1$-planar graph is irreducible if it has 
at least $10$ vertices, and thus $m\le 4n-9$ for $n\ge 10$ (since $G$ is 
reducible). For if $n=9$, then any 1-planar graph with $9$ vertices has at most 
$4\times 9-9$ edges, and thus $m\le 4\times 9-9$ by Lemma \ref{maxedges}.
 Let  
$P_{\frac{n}{3}}$ be a 
path with $\frac{n}{3}$ vertices where $n\ge 9$ and $n \equiv 0~(\bmod 3)$. Let $V(P_{\frac{n}{3}})=\{u_1,u_2,\dots,u_{\frac{n}{3}}\}$.  Let 
$H=P_{\frac{n}{3}}\circ C_3$. Clearly, $H$ is reducible, and $|E(H)|=4n-9$. 
Furthermore, a 1-planar drawing of $H$ is shown in  Fig. \ref{colfig}, where $C_{u_i}$ is a 3-cycle  in $H$ that 
corresponds to  $u_i$ in $P_{\frac{n}{3}}$ for $1\le i\le 
\frac{n}{3}$. Thus, $4n-9$ is a tight bound for every $k\ge 3$ and $n=3k$.
\proofend

\begin{figure}
	\centering

\begin{tikzpicture}[scale=0.45]
	\begin{pgfonlayer}{nodelayer}
		\node [style=blacknode] (0) at (0.75, 6) {};
		\node [style=blacknode] (1) at (-2.65, 0) {};a
		\node [style=blacknode] (2) at (4.25, 0) {};
		\node [style=blacknode] (3) at (0.75, 3.95) {};
		\node [style=blacknode] (4) at (-0.925, 1) {};
		\node [style=blacknode] (5) at (2.45, 1) {};
		\node [style=blacknode] (6) at (-4.5, -1) {};
		\node [style=blacknode] (7) at (6, -1) {};
		\node [style=blacknode] (11) at (0.75, 8) {};
		\node [style=none] (12) at (0.75, 2.5) {$\vdots$};
		\node [style=none] (13) at (1.88, 7) {\scriptsize $C_{u_1}$};
		\node [style=none] (14) at (1.6, 5.5) {\scriptsize $C_{u_2}$};
		\node [style=none] (15) at (1.6, 3.5) {\scriptsize $C_{u_3}$};
	\end{pgfonlayer}
	\begin{pgfonlayer}{edgelayer}
		\draw [style={blue_thick_opacity}] (0) to (3);
		\draw [style={blue_thick_opacity}] (4) to (1);
		\draw [style={blue_thick_opacity}] (5) to (2);
		\draw [style={blue_thick_opacity}] (0) to (4);
		\draw [style={blue_thick_opacity}] (1) to (3);
		\draw [style={blue_thick_opacity}] (3) to (2);
		\draw [style={blue_thick_opacity}] (0) to (5);
		\draw [style={blue_thick_opacity}] (1) to (5);
		\draw [style={blue_thick_opacity}] (4) to (2);
		\draw [style={blue_thick_opacity}] (6) to (0);
		\draw [style={blue_thick_opacity}] (6) to (1);
		\draw [style={blue_thick_opacity}] (6) to (2);
		\draw [style={blue_thick_opacity}] (7) to (1);
		\draw [style={blue_thick_opacity}] (7) to (2);
		\draw [style={blue_thick_opacity}] (7) to (0);
		\draw [style={black_thick}] (0) to (1);
		\draw [style={black_thick}] (0) to (2);
		\draw [style={black_thick}] (2) to (1);
		\draw [style={black_thick}] (3) to (4);
		\draw [style={black_thick}] (3) to (5);
		\draw [style={black_thick}] (4) to (5);
		\draw [style={black_thick}] (11) to (6);
		\draw [style={black_thick}] (6) to (7);
		\draw [style={black_thick}] (11) to (7);
		\draw [style={blue_thick_opacity}] (11) to (1);
		\draw [style={blue_thick_opacity}] (11) to (2);
		\draw [style={blue_thick_opacity}] (11) to (0);
	\end{pgfonlayer}
\end{tikzpicture}

\caption{ A 1-planar drawing of $H=P_{\frac{n}{3}}\circ C_3$}
\label{colfig}
\end{figure}

 \begin{remark}
 We believe that the bound in Theorem \ref{coro1} we provided is still far from resolving Problem \ref{openPro}.
 \end{remark}
\section{Acknowledgment}
We claim that there is no conflict of interest in our paper.
No data was used for the research described in the article.
The work was supported by the National
Natural Science Foundation of China (Grant No. 12271157, 12371346) and the Postdoctoral Science Foundation of China (Grant  No. 2024M760867).

 \nocite{*}
%\bibliographystyle{bib}
%\bibliography{pap.bib}
\end{document}